\documentclass[graybox]{svmult}

\usepackage{type1cm}        % activate if the above 3 fonts are
\usepackage{makeidx}         % allows index generation
\usepackage{graphicx}        % standard LaTeX graphics tool
\usepackage{multicol}        % used for the two-column index
\usepackage[bottom]{footmisc}% places footnotes at page bottom

\usepackage{newtxtext}       % 
\usepackage[varvw]{newtxmath}       % selects Times Roman as basic font

\usepackage{nicefrac}
\usepackage{cite}
\usepackage[group-digits=true]{siunitx}
\usepackage{subcaption}
\usepackage{braket}
\usepackage[algo2e, noline, noend, algoruled]{algorithm2e}
\usepackage{hyperref}

\makeindex             % used for the subject index
\begin{document}

\title*{Shape Optimization with Nonlinear Conjugate Gradient Methods}
\titlerunning{Shape Optimization with Nonlinear Conjugate Gradient Methods}
% Use \titlerunning{Short Title} for an abbreviated version of
% your contribution title if the original one is too long
\author{Sebastian Blauth}
% Use \authorrunning{Short Title} for an abbreviated version of
% your contribution title if the original one is too long
\institute{Sebastian Blauth \at Fraunhofer ITWM, Fraunhofer-Platz 1, 67663 Kaiserslautern \email{sebastian.blauth@itwm.fraunhofer.de}}
%
% Use the package "url.sty" to avoid
% problems with special characters
% used in your e-mail or web address
%
\maketitle

\vspace{-2cm}

{\centering
This is a post-peer-review, pre-copyedit version of an article published in Lecture Notes in Computational Science and Engineering, Volume 137. 
The final version is available online at \url{https://doi.org/10.1007/978-3-031-20432-6_9}.
}

\vspace{2cm}

\abstract*{In this chapter, we investigate recently proposed nonlinear conjugate gradient (NCG) methods for shape optimization problems. We briefly introduce the methods as well as the corresponding theoretical background and investigate their performance numerically. The obtained results confirm that the NCG methods are efficient and attractive solution algorithms for shape optimization problems.}

\abstract{In this chapter, we investigate recently proposed nonlinear conjugate gradient (NCG) methods for shape optimization problems. We briefly introduce the methods as well as the corresponding theoretical background and investigate their performance numerically. The obtained results confirm that the NCG methods are efficient and attractive solution algorithms for shape optimization problems.}

\section{Introduction}
\label{sblauth:sec:intro}

Shape optimization problems arise in many industrial applications, such as the design optimization of airplanes \cite{Schmidt2013Three}, automobiles \cite{Othmer2014Adjoint}, electric motors \cite{Gangl2015Shape}, microchannel systems \cite{Blauth2020Shape, Blauth2020Model}, polymer spin packs \cite{Hohmann2019Shape, Leithaeuser2018Designing}, and melting furnaces \cite{Leithaeuser2021Energy}. 
For the solution of such problems, shape optimization based on shape calculus (see, e.g., \cite{Delfour2011Shapes}) has attracted lots of research interest in recent years, with particular regards to the development of efficient solution algorithms. This can be seen, e.g., in \cite{Schulz2014Riemannian} and \cite{Schulz2016Efficient}, where Newton and limited memory BFGS (L-BFGS) methods for shape optimization are proposed, respectively, and in \cite{Etling2020First, Mueller2021novel}, where special mesh deformation procedures for increasing the mesh quality and the avoidance of remeshing are investigated.

In this chapter, we consider the recently proposed nonlinear conjugate gradient (NCG) methods for shape optimization from \cite{Blauth2021Nonlinear}, which are numerical solution algorithms for such problems, and investigate their performance. One of the benefits of these NCG methods is that they require only slightly more memory than the popular gradient descent method, while being significantly more efficient. Here, we compare the numerical performance of the NCG methods to the already established gradient descent and L-BFGS methods for shape optimization. For this numerical comparison we utilize our software package cashocs \cite{Blauth2021cashocs} which implements these methods and allows for a detailed comparison.

This chapter is structured as follows. In Sect.~\ref{sblauth:sec:theory}, we briefly present some theoretical background on shape optimization and shape calculus. This is required for the presentation of the nonlinear conjugate gradient methods in Sect.~\ref{sblauth:sec:ncg}. Finally, the numerical examples, which showcase the capabilities of the NCG methods, can be found in Sect.~\ref{sblauth:sec:numerics}.

\section{Theoretical Background}
\label{sblauth:sec:theory}

In this section, we present some theoretical background on shape optimization and shape calculus which we require for our presentation of the NCG methods in Sect.~\ref{sblauth:sec:ncg}.

\subsection{Fundamentals of Shape Optimization}
\label{sblauth:subsec:shape_opt}

A general shape optimization problem with a partial differential equation (PDE) as constraint is given by
\begin{equation}
	\label{sblauth:eq:sop}
	\begin{aligned}
		&\min_{\Omega, u} \mathcal{J}(\Omega, u) \\
		&\text{subject to} 
		&\quad \begin{aligned}[t]
			&e(\Omega, u) = 0, \\
			&\Omega \in \mathcal{A}.
		\end{aligned}
	\end{aligned}
\end{equation}
Here, $\mathcal{J}$ is a cost functional which we want to minimize over a set of admissible geometries $\mathcal{A} \subset \Set{\Omega \subset D}$ for some given bounded hold-all domain $D\subset \mathcal{R}^d$ and $u$ is the so-called state variable which lies in the state space $U(\Omega)$. Additionally, $e(\Omega, \cdot) \colon U(\Omega) \to V(\Omega)^*$, where $V(\Omega)^*$ is the dual space of $V(\Omega)$, is an operator that models the PDE constraint, which we interpret in the weak form
\begin{equation}
	\label{sblauth:eq:weak_state}
	\text{Find } u\in U(\Omega) \text{ such that} \quad \langle e(\Omega, u), v\rangle_{V(\Omega)*, V(\Omega)} = 0 \quad \text{ for all } v\in V(\Omega).
\end{equation}
We assume that the state equation \eqref{sblauth:eq:weak_state} admits a unique solution $u=u(\Omega)$ for all $\Omega \in \mathcal{A}$ so that we have $e(\Omega, u(\Omega)) = 0$. With this, we can introduce the reduced cost functional $J(\Omega) = \mathcal{J}(\Omega, u(\Omega))$ and rephrase \eqref{sblauth:eq:sop} equivalently as
\begin{equation}
	\label{sblauth:eq:reduced_shape}
	\min_{\Omega} J(\Omega) \quad \text{ subject to } \quad \Omega \in \mathcal{A},
\end{equation}
where we have formally eliminated the PDE constraint.

\subsection{Shape Calculus}

To obtain efficient solution algorithms for problem~\eqref{sblauth:eq:sop}, one can use techniques from shape calculus to derive sensitivities of the cost functional $J(\Omega)$ w.r.t. variations of the domain $\Omega$. In the following, we briefly recall these techniques and we refer to, e.g., \cite{Delfour2011Shapes} for an exhaustive treatment of this topic. 

We define a family of deformed domains $\Omega_t$ for $t\geq 0$ as
\begin{equation*}
	\Omega_t = F_t(\Omega) = \Set{F_t(x) | x\in \Omega}.
\end{equation*}
In this chapter, we use the so-called perturbation of identity (cf.\ \cite{Delfour2011Shapes}), in which the transformation $F_t$ is given by
\begin{equation*}
	F_t(x) = (I + t\mathcal{V})(x),
\end{equation*}
where $I$ denotes the identity matrix in $\mathbb{R}^d$ and $\mathcal{V}$ is a vector field in $C^k_0(D;\mathbb{R}^d)$ for some $k\geq 1$, i.e., the space of $k$-times continuously differentiable functions from $D$ to $\mathbb{R}^d$ with compact support. Note, that there exist other, equivalent approaches for calculating first order shape derivatives, such as the speed method (cf.\ \cite{Delfour2011Shapes}), but for the sake of brevity we only consider the perturbation of identity in this chapter. Now, we can define the shape derivative as follows (cf.\ \cite[Chapter 9, Definition 3.4]{Delfour2011Shapes}).
\begin{definition}
	Let $\tau > 0$ be sufficiently small, $\mathcal{A} \subset \Set{\Omega \subset D}$, $J\colon \mathcal{A} \to \mathbb{R}$, and $\Omega \in \mathcal{A}$. Additionally, let $\mathcal{V}\in C^k_0(D;\mathbb{R}^d)$ with $k\geq 1$, let $F_t = I + t\mathcal{V}$ be the perturbation of identity with $\mathcal{V}$, and assume that $\Omega_t = F_t(\Omega) \in \mathcal{A}$ for all $t\in [0, \tau]$.
	
	We say that $J$ has a Eulerian semi-derivative at $\Omega$ in direction $\mathcal{V}$ if the limit
	\begin{equation*}
		dJ(\Omega)[\mathcal{V}] := \lim\limits_{t\searrow 0} \frac{J(F_t(\Omega)) - J(\Omega)}{t} = \left. \frac{d}{dt} J(F_t(\Omega)) \right\rvert_{t=0^+}
	\end{equation*}
	exists. Moreover, let $\Xi$ be a topological vector subspace of $C^\infty_0(D;\mathbb{R}^d)$. We say that $J$ is shape differentiable at $\Omega$ w.r.t.\ $\Xi$ if it has a Eulerian semi-derivative at $\Omega$ in all directions $\mathcal{V} \in \Xi$ and, additionally, the mapping
	\begin{equation*}
			dJ(\Omega) \colon \Xi \to \mathbb{R}; \quad \mathcal{V} \mapsto dJ(\Omega)[\mathcal{V}]
	\end{equation*}
	is linear and continuous. In this case, we call $dJ(\Omega)[\mathcal{V}]$ the shape derivative of $J$ at $\Omega$ w.r.t.\ $\Xi$ in direction $\mathcal{V} \in \Xi$.
\end{definition}

Note, that there exist several possibilities for calculating shape derivatives in the context of PDE constrained shape optimization, an overview of which can, e.g., be found in \cite{Sturm2015Shape}. These methods involve solving a so-called adjoint equation to calculate the shape derivative, which is usual in PDE constrained optimization (cf.\ \cite{Hinze2009Optimization}).

An important result from shape calculus is Hadamard's structure theorem, which we briefly recall here.

\begin{theorem}[Structure Theorem]
	\label{sblauth:thm:structure}
	Let J be a shape functional which is shape differentiable at some $\Omega\subset D$ and let $\Gamma = \partial \Omega$ be compact. Further, let $k \geq 0$ be the smallest integer such that $dJ(\Omega) \colon C^\infty_0(D;\mathbb{R}^d) \to \mathbb{R};\ \mathcal{V} \mapsto dJ(\Omega)[\mathcal{V}]$ is continuous w.r.t. the $C^k_0(D;\mathbb{R}^d)$ topology, and assume that $\Gamma$ is of class $C^{k+1}$. Then, there exists a continuous, linear functional $g:C^k(\Gamma) \to \mathbb{R}$ such that
	\begin{equation*}
		dJ(\Omega)[\mathcal{V}] = g[\mathcal{V} \cdot n],
	\end{equation*}
	where $n$ is the outer unit normal vector on $\Gamma$. In particular, if $g\in L^1(\Gamma)$, it holds that
	\begin{equation*}
		dJ(\Omega)[\mathcal{V}] = \int_{\Gamma} g\ \mathcal{V}\cdot n \text{ d}s.
	\end{equation*}
\end{theorem}
\begin{proof}
	The proof can be found in \cite[Theorem~3.6 and Corollary 1, pp. 479--481]{Delfour2011Shapes}.
\end{proof}

\subsection{A Riemannian View on Shape Optimization and Steklov-Poincar\'e-Type Metrics}

In order to formulate the nonlinear CG methods for shape optimization, we now briefly recall the Riemannian view on shape optimization from \cite{Schulz2014Riemannian} as well as the corresponding Steklov-Poincar\'e-type metrics from \cite{Schulz2016Efficient}. 

We consider compact and connected subsets $\Omega \subset D \subset \mathbb{R}^2$ with $C^\infty$ boundary, where $D$ is, again, some bounded hold-all domain. As in \cite{Schulz2014Riemannian}, we define 
\begin{equation*}
	B_e(S^1; \mathbb{R}^2) := \text{Emb}(S^1;\mathbb{R}^2) / \text{Diff}(S^1),
\end{equation*}
i.e., the set of all equivalence classes of $C^\infty$ embeddings of the unit circle $S^1 \subset \mathbb{R}^2$ into $\mathbb{R}^2$, given by $\text{Emb}(S^1;\mathbb{R}^2)$, where the equivalence relation is defined via the set of all $C^\infty$ diffeomorphisms of $S^1$ into itself, given by $\text{Diff}(S^1)$. Note, that this equivalence relation factors out reparametrizations as these do not change the underlying shape. It is shown in \cite{Kriegl1997convenient} that $B_e$, which is the set of all smooth two-dimensional shapes, is a smooth manifold. An element of $B_e(S^1;\mathbb{R}^2)$ is represented by a smooth curve $\Gamma\colon S^1 \to \mathbb{R}^2;\ \theta \mapsto \Gamma(\theta)$. Due to the equivalence relation, the tangent space at $\Gamma \in B_e$ is isomorphic to the set of all $C^\infty$ normal vector fields along $\Gamma$, i.e.,
\begin{equation*}
		T_\Gamma B_e \cong \Set{h | h=\alpha n, \alpha\in C^\infty(\Gamma;\mathbb{R})} \cong \Set{\alpha | \alpha \in C^\infty(\Gamma;\mathbb{R})}.
\end{equation*}

As in \cite{Schulz2016Efficient}, we consider the following Steklov-Poincar\'e-type metric $g^S_\Gamma$ at some $\Gamma \in B_e$, which is defined as
\begin{equation}
	\label{sblauth:eq:def_riemannian}
		g^S_\Gamma \colon H^{\nicefrac{1}{2}}(\Gamma) \times H^{\nicefrac{1}{2}}(\Gamma) \to \mathbb{R};\ (\alpha, \beta) \mapsto \int_{\Gamma} \alpha \left(S^p_\Gamma\right)^{-1} \beta \text{ d}s.
\end{equation}
Here, $S^p_\Gamma$ is a symmetric and coercive operator defined by
\begin{equation*}
		S^p_\Gamma \colon H^{-\nicefrac{1}{2}}(\Gamma) \to H^{\nicefrac{1}{2}}(\Gamma);\ \alpha \mapsto U \cdot n,
\end{equation*}
where $U \in H^1(\Omega)^d$ solves the problem
\begin{equation}
	\label{sblauth:eq:sp_bilinear}
		\text{Find } U \in H^1(\Omega)^d \text{ such that } \quad a_\Omega(U, V) = \int_{\Gamma} \alpha \left(V\cdot n\right) \text{ d}s \quad \text{ for all } V \in H^1(\Omega)^d,
\end{equation}
for a symmetric, continuous, and coercive bilinear form $a_\Omega\colon H^1(\Omega)^d \times H^1(\Omega)^d \to \mathbb{R}$. To define a Riemannian metric on $B_e$, we restrict $g^S$ to the tangent space $T_\Gamma B_e$.

Let us now briefly discuss the relation between the metric $g^S_\Gamma$ and shape calculus. To do so, we assume that the shape functional $J$ is shape differentiable and has a shape derivative of the form
\begin{equation*}
	dJ(\Omega)[\mathcal{V}] = \int_{\Gamma} g\ \mathcal{V}\cdot n \text{ d}s,
	\end{equation*}
with $g \in L^2(\Gamma)$ (cf.\ Theorem~\ref{sblauth:thm:structure}). Then, the Riemannian shape gradient w.r.t.\ $g^S_\Gamma$ is given by $\gamma \in T_\Gamma B_e$, which is the solution of
\begin{equation}
	\label{sblauth:eq:riesz_riemannian}
		\text{Find } \gamma \in T_\Gamma B_e \text{ such that }
		\quad g^S_\Gamma(\gamma, \phi) = \int_{\Gamma} g \phi \text{ d}s \quad \text{ for all } \phi \in T_\Gamma B_e.
\end{equation}
Due to the definition of $g^S_\Gamma$, the solution of \eqref{sblauth:eq:riesz_riemannian} is given by $\gamma = S^p_\Gamma g$, in particular, we have that $\gamma = \mathcal{G} \cdot n$, where $\mathcal{G}$ solves
\begin{equation}
	\label{sblauth:eq:def_gradient_defo}
	\text{Find } \mathcal{G} \in H^1(\Omega)^d \text{ such that } \quad a_\Omega(\mathcal{G}, \mathcal{V}) = dJ(\Omega)[\mathcal{V}] \quad \text{ for all } \mathcal{V} \in H^1(\Omega)^d.
\end{equation}
Due to the Lax-Milgram lemma, this problem has a unique solution $\mathcal{G}$ which we call the gradient deformation of $J$ at $\Omega$. The gradient deformation $\mathcal{G}$ can be interpreted as an extension of the shape gradient $\gamma$ to the entire domain $\Omega$. Since $a_\Omega$ is coercive, there exists a constant $C>0$ so that
\begin{equation}
	\label{chap4:eq:gradient_defo_descent}
	dJ(\Omega)[-\mathcal{G}] = a_\Omega(-\mathcal{G}, \mathcal{G}) \leq -C \left\lvert\lvert \mathcal{G} \right\rvert\rvert_{H^1(\Omega)}^2 \leq 0,
\end{equation}
i.e., an infinitesimal perturbation of identity with the negative gradient deformation yields a descent in the shape functional $J$. This fact is often used for the solution of shape optimization problems with a gradient descent method (see, e.g., \cite{Etling2020First, Hohmann2019Shape}).

% Use the \index{} command to code your index words
%
% For tables use
%

%\begin{table}[!t]
%\caption{Please write your table caption here}
%\label{tab:1}       % Give a unique label
%%
%% Follow this input for your own table layout
%%
%\begin{tabular}{p{2cm}p{2.4cm}p{2cm}p{4.9cm}}
%\hline\noalign{\smallskip}
%Classes & Subclass & Length & Action Mechanism  \\
%\noalign{\smallskip}\svhline\noalign{\smallskip}
%Translation & mRNA$^a$  & 22 (19--25) & Translation repression, mRNA cleavage\\
%Translation & mRNA cleavage & 21 & mRNA cleavage\\
%Translation & mRNA  & 21--22 & mRNA cleavage\\
%Translation & mRNA  & 24--26 & Histone and DNA Modification\\
%\noalign{\smallskip}\hline\noalign{\smallskip}
%\end{tabular}
%$^a$ Table foot note (with superscript)
%\end{table}

\section{Nonlinear CG Methods for Shape Optimization}
\label{sblauth:sec:ncg}

We now use the theoretical framework introduced in Sect.~\ref{sblauth:sec:theory} to formulate our NCG methods for shape optimization from \cite{Blauth2021Nonlinear}. To do so, we consider only reduced shape optimization problems of the form \eqref{sblauth:eq:reduced_shape}. For the sake of brevity, we only focus on a theoretical description of the methods. However, a detailed description of our numerical implementation of the methods can be found in our previous work \cite{Blauth2021Nonlinear}.

%Before we are able to describe the NCG methods, we briefly recall some basic concepts of optimization on manifolds. For a more details, we refer the reader to, e.g., \cite{Absil2008Optimization, Ring2012Optimization}. For some $\Gamma \in B_e$, a retraction $R_\Gamma$ is a smooth mapping
%\begin{equation*}
%	R_\Gamma \colon T_\Gamma B_e \to B_e;\ \alpha \mapsto R_\Gamma(\alpha),
%\end{equation*}
%which satisfies $R_{\Gamma} (0_\Gamma) = \Gamma$ and $D R_\Gamma (0_\Gamma) = \text{id}_{R_\Gamma B_e}$, where $D R\Gamma$ denotes the derivative of $R_\Gamma$, and $0_\Gamma$ and $\text{id}_{T_\Gamma B_e}$ are the zero element and identity mapping in $T_\Gamma B_e$, respectively. 
%Moreover, we denote by 
%\begin{equation*}
%	\mathcal{T} \colon T B_e \oplus T B_e \to T B_e;\ (\alpha, \beta) \mapsto \mathcal{T}_{\alpha}\ \beta
%\end{equation*}
%a vector transport which satisfies the following properties: For $\alpha, \beta \in T_\Gamma B_e$ it holds that $\mathcal{T}_\alpha\ \beta$ is an element of $T_{R_\Gamma (\alpha)} B_e$. Further, we have the relations $\mathcal{T}_{0_\Gamma} \ \alpha = \alpha$ and $\mathcal{T}_{\gamma} \ (a \alpha + b \beta) = a \mathcal{T}_\gamma \ \alpha + b \mathcal{T}_{\gamma} \ \beta$ for $a, b \in \mathbb{R}$. 
%Note, that the parallel transport or parallel translation, as defined in, e.g., \cite{Absil2008Optimization, Ring2012Optimization}, is a vector transport associated to the exponential map.

Starting from an initial guess $\Omega_0$, the $k$-th iteration of the NCG methods consists of the following steps: First, we compute the shape derivative $dJ(\Omega)[\cdot]$ of our cost functional $J$ which involves solving the state and adjoint equation on the domain $\Omega_k$ (cf.\ \cite{Blauth2021Nonlinear}). Next, we compute the Riemannian shape gradient $\gamma_k$ via equation~\eqref{sblauth:eq:riesz_riemannian}. Note, that this involves the computation of the gradient deformation (cf.\ equation~\eqref{sblauth:eq:def_gradient_defo}), which is beneficial for our numerical implementation of the method as described detailedly in \cite{Blauth2021Nonlinear}. The next step of the methods involves the computation of the search direction $\delta_k$, which is defined as
\begin{equation*}
	\delta_k = -\gamma_k + \beta_k \mathcal{T}_{\eta_{k-1}} \delta_{k-1},
\end{equation*}
where we set $\delta_{0} = -\gamma_0$. Here, the parameter $\eta_{k-1}$ is defined as $\eta_{k-1} = t_{k-1}\delta_{k-1}$ with the previous step size $t_{k-1}$ (see below) and $\mathcal{T}$ denotes a vector transport (see, e.g., \cite{Absil2008Optimization, Ring2012Optimization} for more details). Moreover, the parameter $\beta_k$ is an update parameter for the NCG method, which is detailed below. After we have obtained the search direction, we have to compute a feasible step size for updating our domain. This can be accomplished with, e.g., an Armijo line search (cf.\ \cite{Blauth2021Nonlinear} for more details) and yields the step size $t_k$ for iteration $k$. The geometry is then updated with the help of an retraction $R$ (see, e.g., \cite{Absil2008Optimization, Ring2012Optimization} for more details), i.e.,
\begin{equation*}
	\Gamma_{k+1} = R_{\Gamma_k} \eta_k,
\end{equation*}
where $\eta_k$ is the scaled search direction for iteration $k$, i.e., $\eta_k = t_k \delta_k$. Finally, we increment $k$ and proceed as before, until an appropriate stopping criterion is reached (cf. \cite{Blauth2021Nonlinear}). Note, that an algorithmic description of this procedure is given in Algorithm~\ref{sblauth:algo:ncg}.

\begin{algorithm2e}[!t]
	\KwIn{Initial geometry $\Omega_0$, represented by its boundary $\Gamma_0$}
	\For{k=0,1,2,\dots, $k_\text{max}$}{
		Compute the shape derivative $dJ(\Omega_k)[\cdot]$ \\
		Compute the Riemannian shape gradient $\gamma_{k}$ by solving \eqref{sblauth:eq:riesz_riemannian} \\
		\If{Stopping criterion is satisfied}{
			Stop with approximate solution $\Gamma_{k}$
		}
		Compute the search direction $\delta_{k} = -\gamma_k + \beta_k \mathcal{T}_{\eta_{k-1}} \delta_{k-1}$ \\
		Compute a feasible step size $t_{k}$ \\
		Set $\eta_k = t_k \delta_k$ and update the geometry $\Gamma_{k+1} = R_{\Gamma_k} \eta_k$
	}
	\caption{Nonlinear CG methods for shape optimization.}
	\label{sblauth:algo:ncg}
\end{algorithm2e}

We consider five different NCG variants, namely the NCG methods of Fletcher and Reeves (FR) \cite{Fletcher1964Function}, Polak and Ribiere (PR) \cite{Polak1969Note}, Hestenes and Stiefel (HS) \cite{Hestenes1952Methods}, Dai and Yuan (DY) \cite{Dai1999nonlinear}, and Hager and Zhang (HZ) \cite{Hager2005new}. Note, that an overview over these methods for finite-dimensional problems can be found, e.g., in \cite{Hager2006survey}. The update parameters $\beta$ in the context of shape optimization are given by
\allowdisplaybreaks
\begin{align*}
		\beta^\mathrm{FR}_k &= \frac{g^S_{\Gamma_k}\left(\gamma_k , \gamma_k \right)}{g^S_{\Gamma_k} \left( \mathcal{T}_{\eta_{k-1}} \gamma_{k-1},\mathcal{T}_{\eta_{k-1}} \gamma_{k-1}  \right)}, \\
		\beta^\mathrm{PR}_k &= \frac{g^S_{\Gamma_k} \left( \gamma_k , y_{k-1} \right) }{g^S_{\Gamma_k} \left( \mathcal{T}_{\eta_{k-1}} \gamma_{k-1},\mathcal{T}_{\eta_{k-1}} \gamma_{k-1}  \right)}, \\
		\beta^\mathrm{HS}_k &= \frac{g^S_{\Gamma_k} \left( \gamma_k , y_{k-1} \right) }{g^S_{\Gamma_k} \left( \mathcal{T}_{\eta_{k-1}} d_{k-1}, y_{k-1} \right)}, \\
		\beta^\mathrm{DY}_k &= \frac{g^S_{\Gamma_k}\left(\gamma_k , \gamma_k \right)}{g^S_{\Gamma_k} \left( \mathcal{T}_{\eta_{k-1}} d_{k-1}, y_{k-1} \right)}, \\
		\beta^\mathrm{HZ}_k &= g^S_{\Gamma_k}\left( y_{k-1} - 2\mathcal{T}_{\eta_{k-1}} d_{k-1} \frac{g^S_{\Gamma_k} \left(y_{k-1} ,y_{k-1} \right)}{g^S_{\Gamma_k}\left( \mathcal{T}_{\eta_{k-1}} d_{k-1}, y_{k-1} \right) } , \frac{\gamma_k}{g^S_{\Gamma_k}\left( \mathcal{T}_{\eta_{k-1}} d_{k-1}, y_{k-1} \right) } \right),
\end{align*}
where we use
\begin{equation*}
	y_{k-1} = \gamma_k - \mathcal{T}_{\eta_{k-1}} \gamma_{k-1}.
\end{equation*}

A particular advantage of the NCG methods is the following: The NCG methods only require one or two additional vectors of storage compared to the popular gradient descent method, while being substantially more efficient, as is shown in Sect.~\ref{sblauth:sec:numerics}. The L-BFGS methods with memory size $m$, on the other hand, require $2m$ additional vectors of storage, which can be prohibitive for very large scale problems, such as the ones arising from industrial applications (see \cite{Kelley1999Iterative}). Hence, the NCG methods are particularly interesting for these kinds of problems, where memory requirements are of great importance.

%In particular, the NCG methods only require one or two more vectors of storage (for the previous gradient $\gamma_{k-1}$ and the previous search direction $d_{k-1}$) compared to the gradient descent method. On the other hand, the limited memory BFGS methods from \cite{Schulz2016Efficient} require $2m$ additional vectors of storage if a memory size of $m$ is used. Therefore, the NCG methods only have slightly higher memory requirements compared to the gradient descent method, but require substantially less memory than the L-BFGS methods. Hence, the NCG methods could be of great interest for the shape optimization of industrial problems, where memory requirements are an issue.
%In the next section, we compare the NCG methods numerically with the gradient descent and the L-BFGS methods.

\section{Numerical Examples}
\label{sblauth:sec:numerics}

In this section, we investigate the previously introduced NCG methods numerically on two benchmark problems. In Sect.~\ref{sblauth:ssec:poisson}, we consider a two-dimensional shape optimization problem with a Poisson equation as PDE constraint and in Sect.~\ref{sblauth:ssec:flow}, we consider the drag minimization in a three-dimensional pipe. For both test cases, we compare the five NCG variants from Sect.~\ref{sblauth:sec:ncg} to the gradient descent and L-BFGS methods. The numerical implementation is done in our software cashocs \cite{Blauth2021cashocs}, which is based on the finite element software FEniCS \cite{Logg2012Automated, Alnes2015FEniCS}, and we refer to our previous work \cite{Blauth2021Nonlinear} for a comprehensive description of our implementation of the NCG methods.

\subsection{Shape Optimization with a Poisson Equation}
\label{sblauth:ssec:poisson}

The first test case is taken from \cite{Etling2020First, Blauth2021Nonlinear} and is given by 
\begin{equation}
	\label{sblauth:eq:poisson}
	\begin{aligned}
		&\min_\Omega J(\Omega, u) = \int_\Omega u \text{ d}x\\
		&\text{subject to } \quad \begin{alignedat}[t]{2}
				-\Delta u &= f \quad &&\text{ in } \Omega, \\
				u &= 0 \quad &&\text{ on } \Gamma,
			\end{alignedat}
	\end{aligned}
\end{equation}
where we consider the problem in two dimensions and use
\begin{equation*}
	f(x) = 2.5 \left( x_1 + 0.4 - x_2^2 \right)^2 + x_1^2 + x_2^2 - 1.
\end{equation*}
Our initial guess $\Omega_0$ is given by the unit disc in $\mathbb{R}^2$. We discretize the PDE constraint with piecewise linear Lagrange elements, for which we use a uniform mesh consisting of \num{7651} nodes and \num{15000} triangles.

We solve this problem with the gradient descent (GD) method, a L-BFGS method with memory size $m=5$ (L-BFGS 5), and the five NCG methods presented in Sect.~\ref{sblauth:sec:ncg}. The history of the optimization can be seen in Fig.~\ref{sblauth:fig:poisson_history}, where the evolution of the cost functional (Fig.~\ref{sblauth:fig:poisson_functional}) and relative gradient norm (Fig.~\ref{sblauth:fig:poisson_norm}) are shown. Here, we have highlighted the graphs of the gradient descent, L-BFGS 5, and the NCG variant of Dai and Yuan (NCG DY), as the latter performed best of all NCG methods. For the sake of better readability, the remaining NCG methods are shown in transparent colors. Here, we observe that all NCG methods perform significantly better than the gradient descent method, as they reach the optimal function value faster and also have lower gradient norms throughout the optimization. On average, the NCG methods achieve a gradient norm that is one order of magnitude smaller than the gradient norm obtained with the gradient descent method. However, the performance of the L-BFGS 5 method is still slightly better than that of the NCG methods, but this comes at the cost of a higher memory usage, as discussed previously.

\begin{figure}[b]
	\centering
	\begin{subfigure}[t]{0.49\textwidth}
		\centering
		\includegraphics[width=\textwidth]{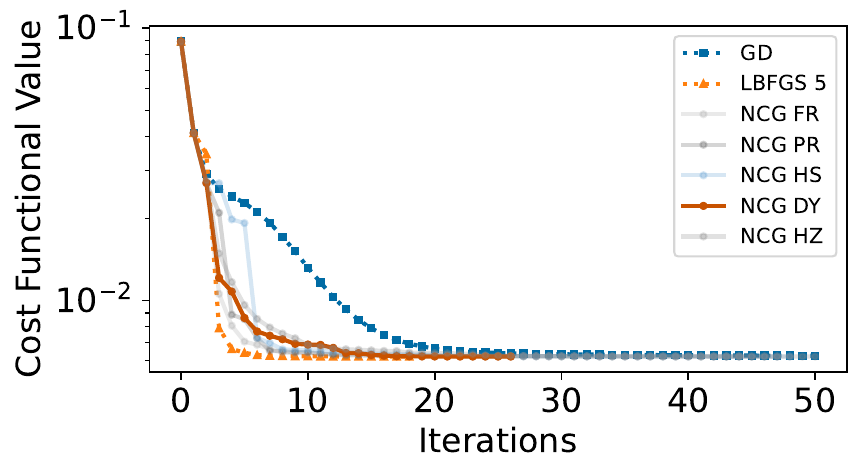}
		\caption{Cost functional (shifted by +0.1).}
		\label{sblauth:fig:poisson_functional}
	\end{subfigure}
	\hfil
	\begin{subfigure}[t]{0.49\textwidth}
		\centering
		\includegraphics[width=\textwidth]{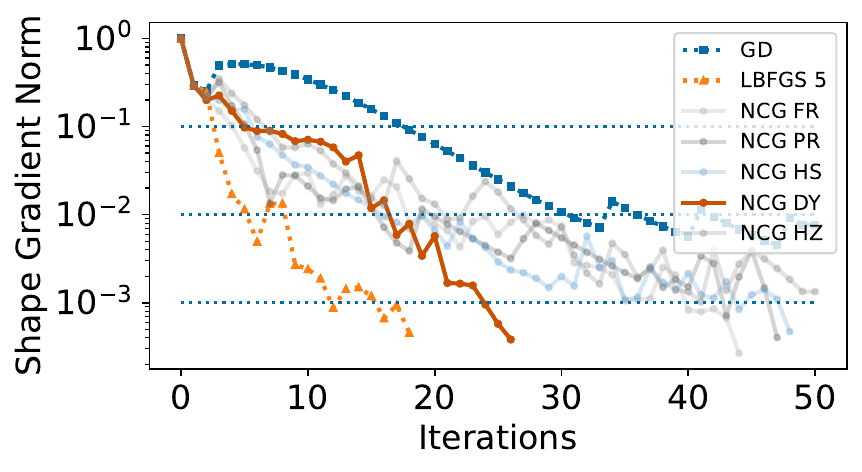}
		\caption{Relative gradient norm.}
		\label{sblauth:fig:poisson_norm}
	\end{subfigure}
	\caption{History of the optimization methods for problem~\eqref{sblauth:eq:poisson}.}
	\label{sblauth:fig:poisson_history}
\end{figure}

\begin{figure}[t]
	\centering
	\begin{subfigure}[t]{0.275\textwidth}
		\includegraphics[width=1\textwidth]{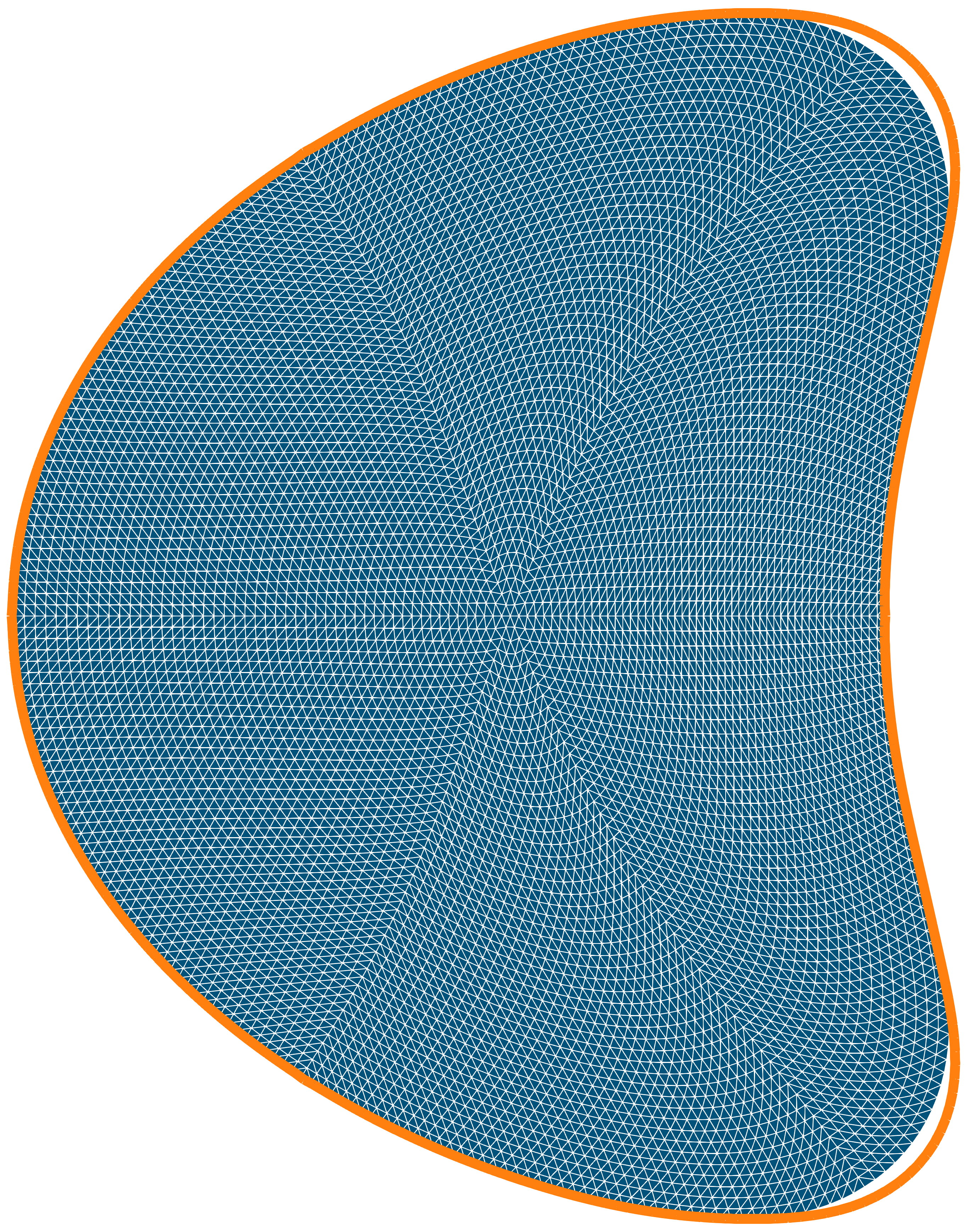}
		\caption{Gradient descent method.}
	\end{subfigure}
	\hfil
	\begin{subfigure}[t]{0.275\textwidth}
		\includegraphics[width=1\textwidth]{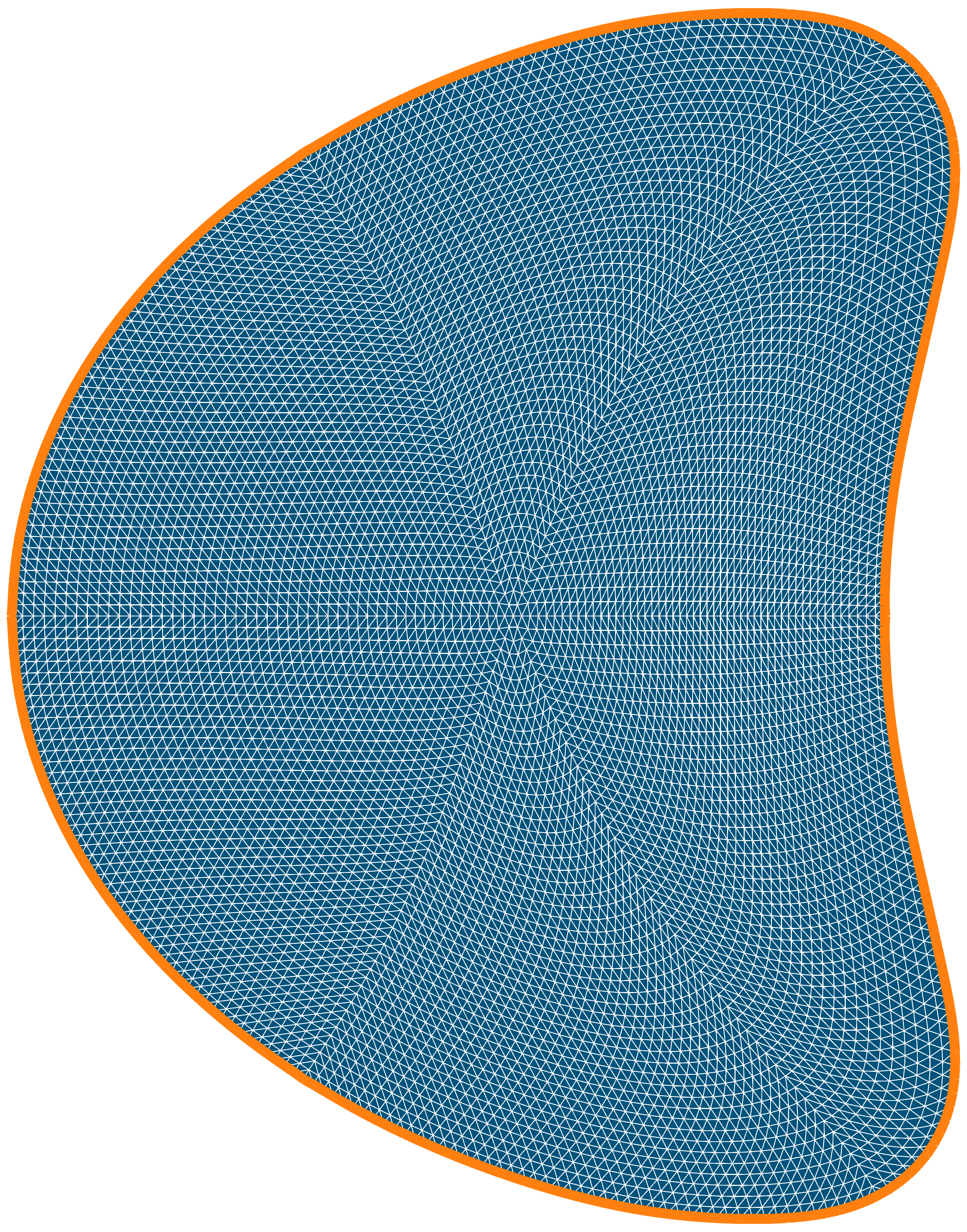}
		\caption{Fletcher-Reeves NCG method.}
	\end{subfigure}
	\hfil
	\begin{subfigure}[t]{0.275\textwidth}
		\includegraphics[width=1\textwidth]{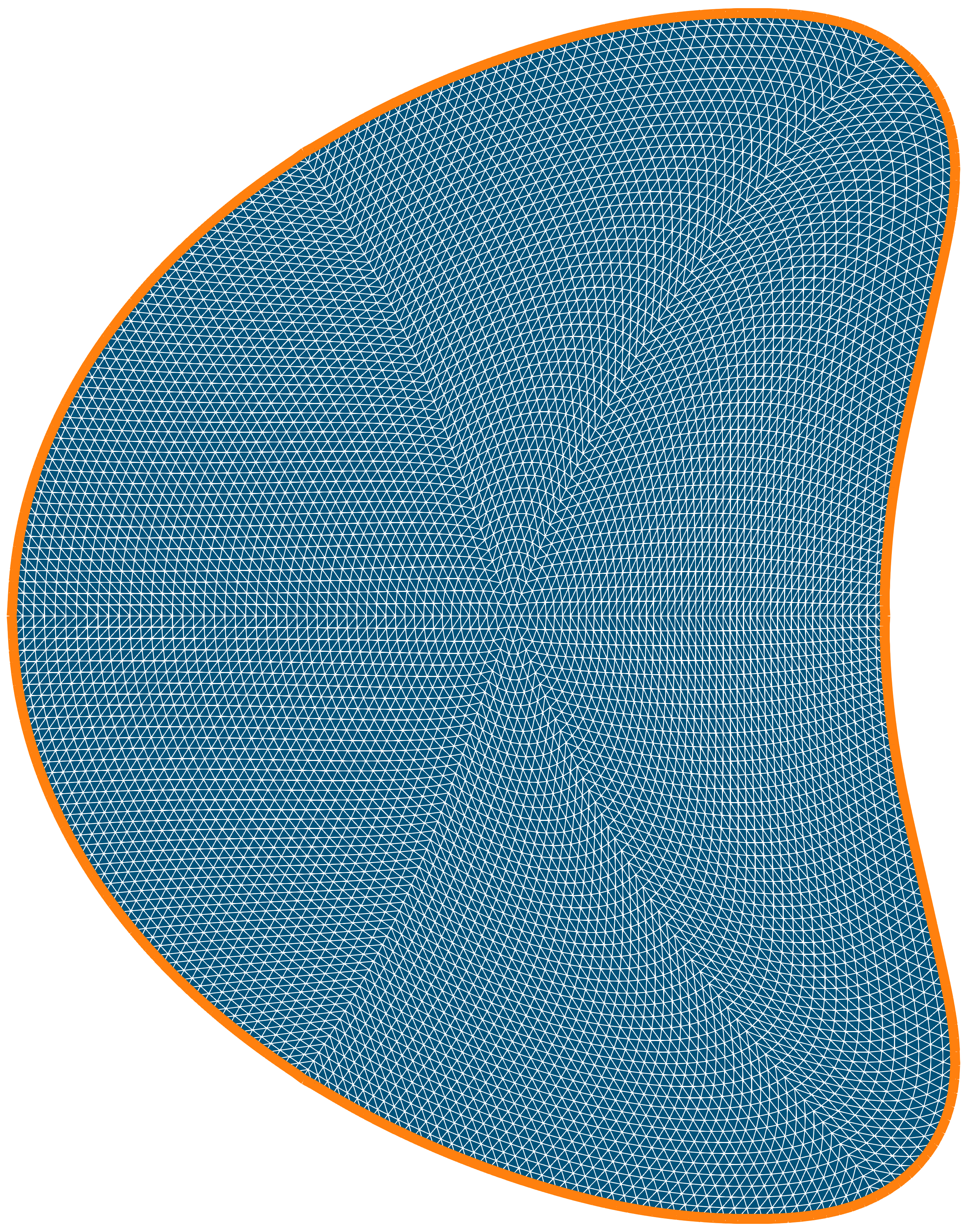}
		\caption{Polak-Ribi\`ere NCG method.}
	\end{subfigure}
	\caption{Optimized Shapes (blue) compared to the reference solution (orange) for the Poisson problem~\eqref{sblauth:eq:poisson}.}
	\label{sblauth:fig:poisson_geom}
\end{figure}

A visual inspection of the optimized geometries, which can be found in Fig.~\ref{sblauth:fig:poisson_geom}, shows that the NCG methods perform very well. Here, the optimized geometries for the gradient descent, Fletcher-Reeves NCG and Polak-Ribiere NCG methods are shown and compared to the reference solution (computed with the L-BFGS 5 method). We observe, that there are still visible differences between the reference solution and the one obtained with the gradient descent method, particularly in the right corners of the geometry. The two NCG methods, however, show no visible differences compared to the reference domain, which underlines their improved convergence behavior compared to the gradient descent method.

\subsection{Shape Optimization of Energy Dissipation in a Pipe}
\label{sblauth:ssec:flow}

For our second test case, we consider the problem of reducing the energy dissipation of a fluid in a three-dimensional pipe, which is taken from \cite{Paganini2021Fireshape}. Here, the flow of the fluid is governed by the Navier-Stokes equations. The corresponding shape optimization problem is given by
\begin{equation}
	\label{sblauth:eq:navier_stokes}
	\begin{aligned}
		&\min_\Omega\ J(\Omega, u) = \frac{1}{\mathrm{Re}} \int_\Omega \varepsilon(u) : \varepsilon(u) \text{ d}x + \frac{\gamma}{2} \left( \int_\Omega 1 \text{ d}x - \int_{\Omega_0} 1 \text{ d}x \right)^2 \\
		&\text{subject to } \quad \begin{alignedat}[t]{2}
			-\frac{2}{\mathrm{Re}} \nabla \cdot \varepsilon(u) + \left(u\cdot \nabla\right) u + \nabla p &= 0 \quad &&\text{ in } \Omega, \\
			\nabla \cdot u &= 0 \quad &&\text{ in } \Omega, \\
			u &= u^\mathrm{in} \quad &&\text{ on } \Gamma^\mathrm{in},\\
			u &= 0 \quad &&\text{ on } \Gamma^\mathrm{wall},\\
			\frac{2}{\mathrm{Re}} \varepsilon(u) n - pn &= 0 \quad &&\text{ on } \Gamma^\mathrm{out}.
		\end{alignedat}
	\end{aligned}
\end{equation}
\begin{figure}[t]
	\centering
	\begin{subfigure}[t]{0.49\textwidth}
		\centering
		\includegraphics[width=\textwidth]{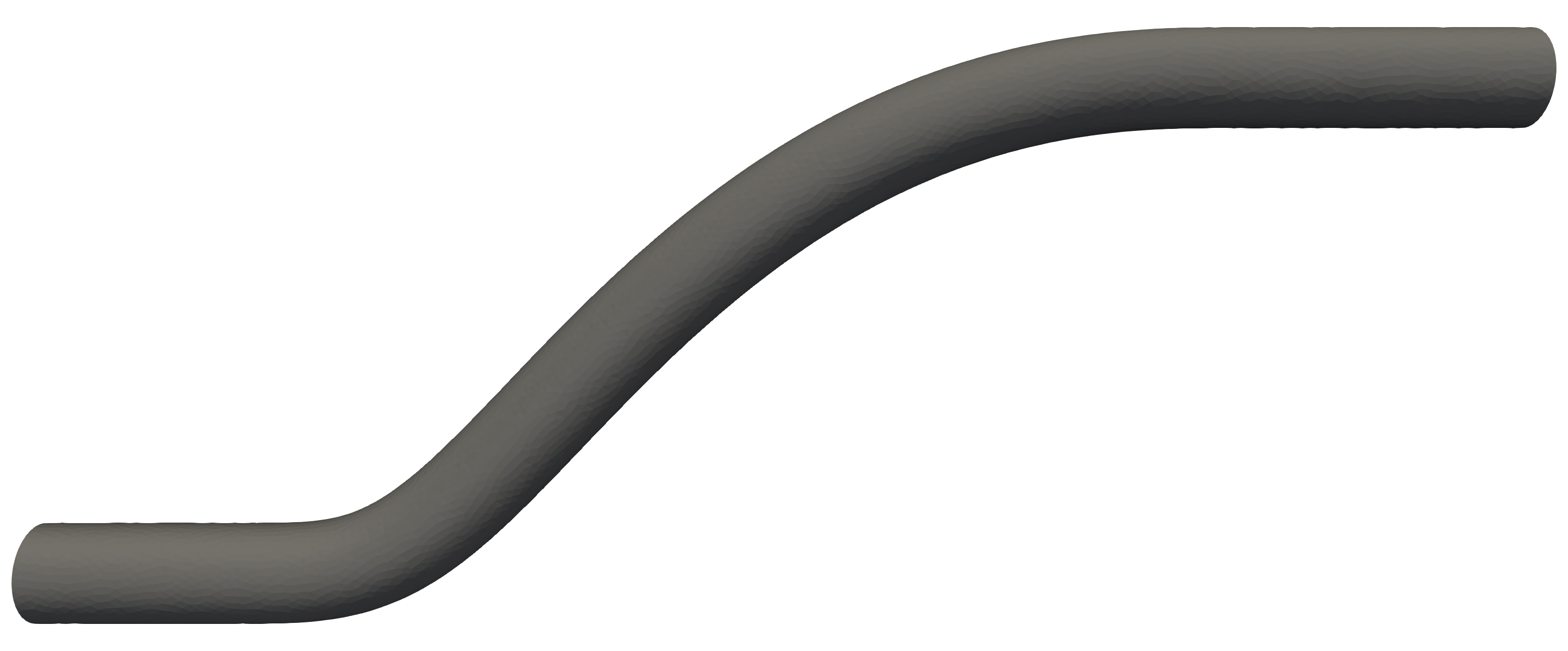}
		\caption{Initial geometry of the pipe.}
		\label{sblauth:fig:initial_pipe}
	\end{subfigure}
	\hfil
	\begin{subfigure}[t]{0.49\textwidth}
		\centering
		\includegraphics[width=\textwidth]{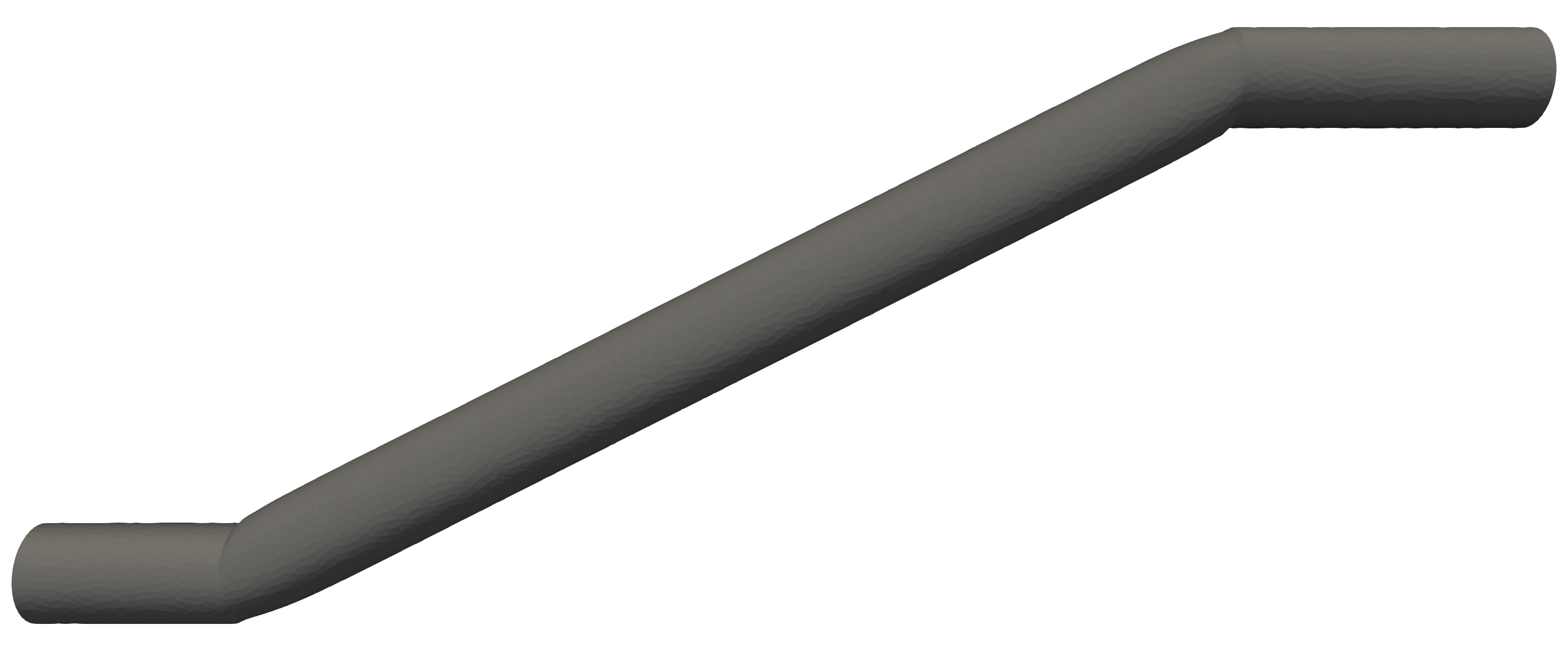}
		\caption{Optimized geometry of the pipe (obtained with the NCG HS method).}
		\label{sblauth:fig:optimized_pipe}
	\end{subfigure}
	\caption{Initial and optimized geometries for problem \eqref{sblauth:eq:navier_stokes}.}
	\label{sblauth:fig:geometries}
\end{figure}
\begin{figure}
	\centering
	\begin{subfigure}[t]{\textwidth}
		\centering
		\includegraphics[width=0.8\textwidth]{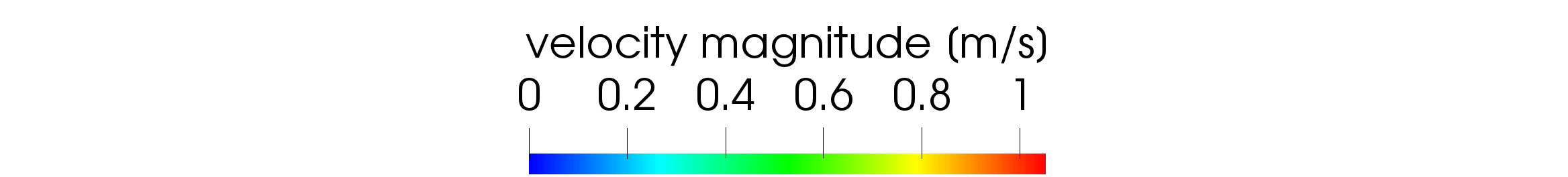}
	\end{subfigure}
	\\[0.25cm]
	\begin{subfigure}[t]{0.49\textwidth}
		\centering
		\includegraphics[width=\textwidth]{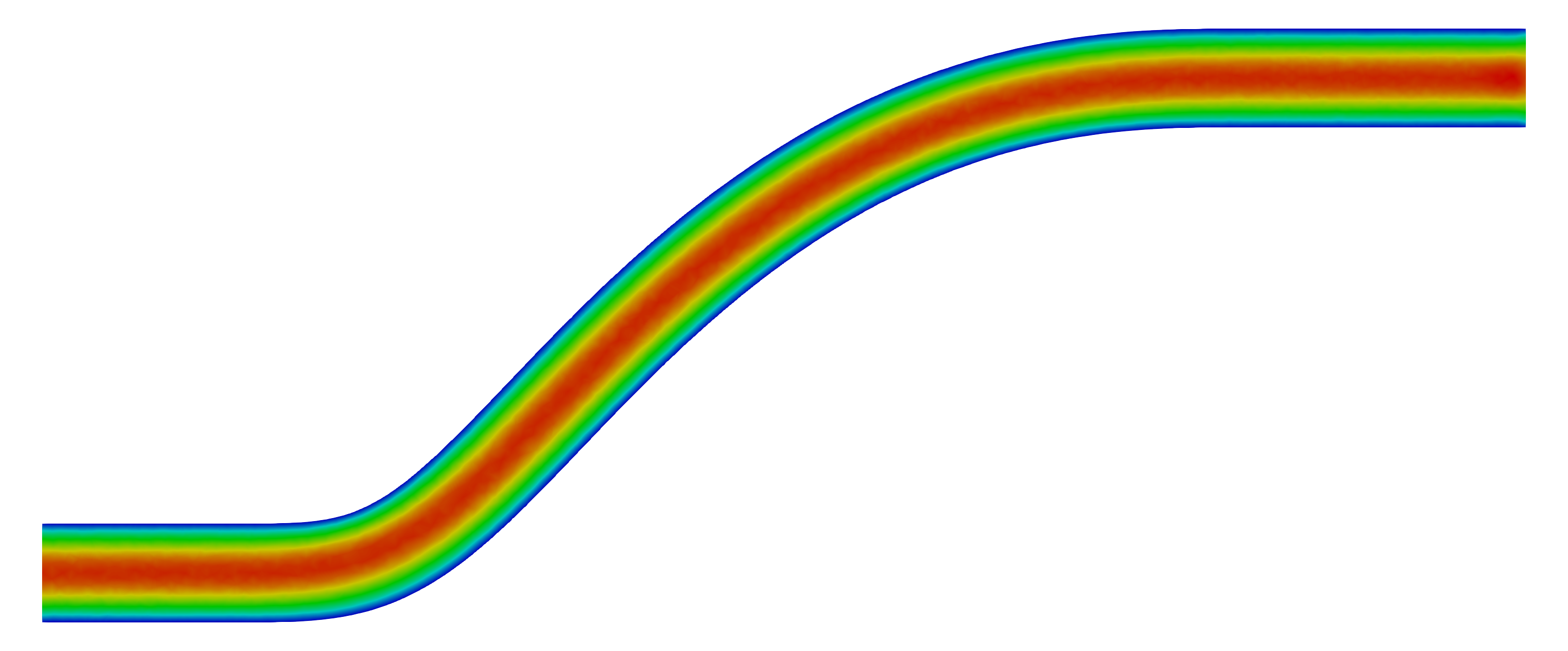}
		\caption{Velocity magnitude on the initial geometry.}
		\label{sblauth:fig:pipe_velo_initial}
	\end{subfigure}
	\hfil
	\begin{subfigure}[t]{0.49\textwidth}
		\centering
		\includegraphics[width=\textwidth]{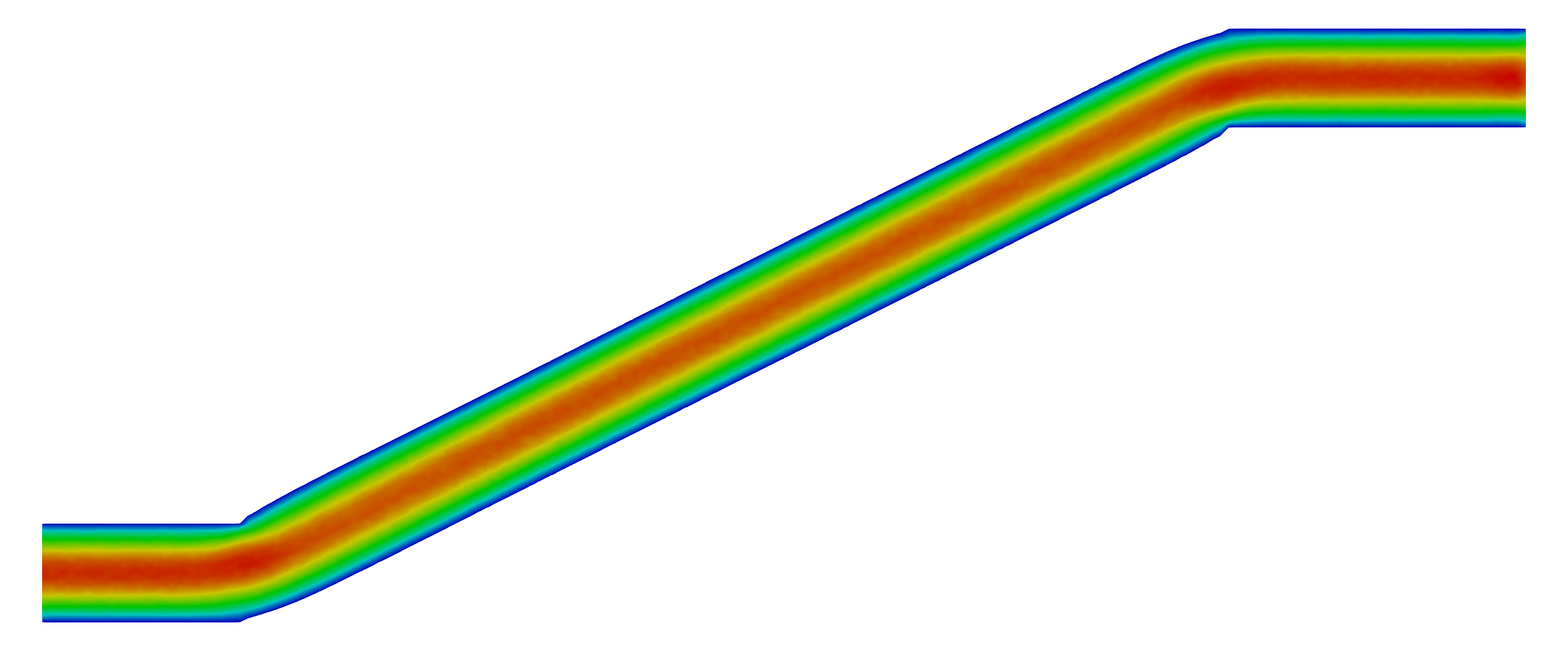}
		\caption{Velocity magnitude on the optimized geometry (obtained with the NCG HS method).}
		\label{sblauth:fig:pipe_velo_optimized}
	\end{subfigure}
	\caption{Velocity magnitude on the initial and optimized geometries, shown as slice through the middle of the geometry.}
	\label{sblauth:fig:pipe_velo}
\end{figure}
Here, $u$ denotes the flow velocity, $p$ the pressure and $\varepsilon(u) = \nicefrac{1}{2} (\nabla u + \nabla u^T)$ is the symmetric gradient of $u$. Note, that the cost functional consists of two terms, where the first one measures the energy dissipation in the pipe and the second one is a regularization of a volume equality constraint. The latter is used to constrain the volume of the pipe to its initial volume. Moreover, the pipe's boundary $\Gamma$ is partitioned into the inlet $\Gamma^\mathrm{in}$ and the wall boundary $\Gamma^\mathrm{wall}$, where we use Dirichlet boundary conditions, as well as the outlet $\Gamma^\mathrm{out}$, where we use a do-nothing boundary condition. We discretize the geometry with \num{17873} nodes and \num{82422} tetrahedrons. Additionally, we discretize the Navier-Stokes equations with the inf-sup-stable Taylor-Hood elements, i.e., piecewise quadratic Lagrange elements for the velocity and piecewise linear Lagrange elements for the pressure. Note, that a plot of the initial and optimized geometries is shown in Fig.~\ref{sblauth:fig:geometries}, and that the velocity magnitude on these domains is visualized in Fig.~\ref{sblauth:fig:pipe_velo}.

As before, we solve this shape optimization problem with the gradient descent, L-BFGS 5, and the NCG methods, where we consider the case of laminar flow and use a Reynolds number of $\mathrm{Re} = 1$ as well as a penalty parameter of $\gamma = 100$. Note, that the choice of $\gamma$ is sufficiently large to ensure a relative volume difference below 0.5~\% between the initial and optimized geometries for all methods, so that the equality constraint is satisfied numerically. The corresponding results of the optimization are shown in Fig.~\ref{sblauth:fig:pipe_history}, where, again, the history of the cost functional (cf.\ Fig.~\ref{sblauth:fig:pipe_functional}) and the relative gradient norm (cf.\ Fig.~\ref{sblauth:fig:pipe_norm}) are shown. Here, we again observe that the NCG methods are very efficient. The best performing NCG method is, again, the one of Dai and Yuan (DY), which even slightly outperformed the L-BFGS 5 method. Additionally, the method of Hestenes and Stiefel (HS) also performed very well and was only slightly worse than the L-BFGS method. The remaining NCG variants performed a bit worse, but all of them were substantially better than the gradient descent method as they required less than half the amount of iterations to reach the prescribed relative tolerance for this problem.

\begin{figure}[t]
	\centering
	\begin{subfigure}[t]{0.49\textwidth}
		\centering
		\includegraphics[width=\textwidth]{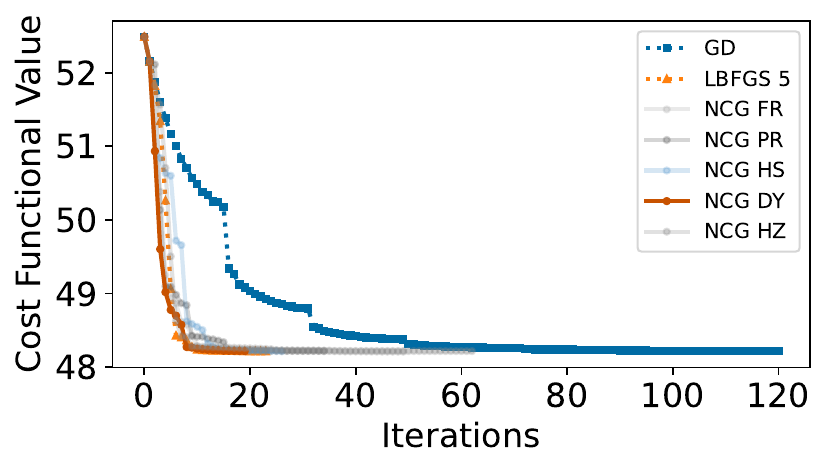}
		\caption{Evolution of the cost functional.}
		\label{sblauth:fig:pipe_functional}
	\end{subfigure}
	\hfil
	\begin{subfigure}[t]{0.49\textwidth}
		\centering
		\includegraphics[width=\textwidth]{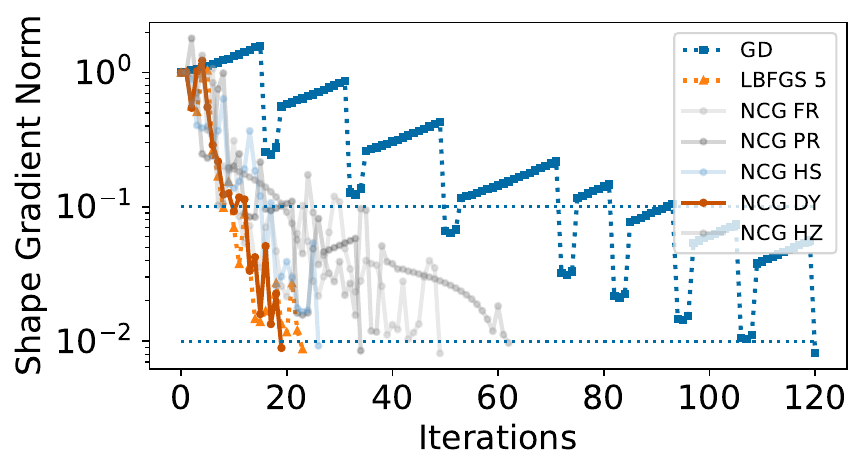}
		\caption{Evolution of the relative gradient norm.}
		\label{sblauth:fig:pipe_norm}
	\end{subfigure}
	\caption{History of the optimization methods for problem~\eqref{sblauth:eq:navier_stokes}.}
	\label{sblauth:fig:pipe_history}
\end{figure}

\section{Conclusions}
\label{sblauth:sec:conclusion}

In this chapter, we have presented and investigated the nonlinear conjugate gradient (NCG) methods for shape optimization from \cite{Blauth2021Nonlinear}. After recalling recent results from shape optimization and shape calculus, we formulated the NCG methods in the Riemannian setting for shape optimization introduced in \cite{Schulz2014Riemannian, Schulz2016Efficient}. Afterwards, we investigated these methods numerically and compared them to the already established gradient descent and L-BFGS methods for shape optimization. The results show that the NCG methods perform substantially better than the popular gradient descent method and that their performance is comparable to the one of the L-BFGS methods from \cite{Schulz2016Efficient}. Hence, the NCG methods could be particularly interesting for large-scale industrial problems due to their efficiency and low memory requirements.

\bibliographystyle{spmpsci}
\bibliography{literature_db.bib}

\begin{thebibliography}{10}
\providecommand{\url}[1]{{#1}}
\providecommand{\urlprefix}{URL }
\expandafter\ifx\csname urlstyle\endcsname\relax
  \providecommand{\doi}[1]{DOI~\discretionary{}{}{}#1}\else
  \providecommand{\doi}{DOI~\discretionary{}{}{}\begingroup
  \urlstyle{rm}\Url}\fi

\bibitem{Absil2008Optimization}
Absil, P.A., Mahony, R., Sepulchre, R.: Optimization algorithms on matrix
  manifolds.
\newblock Princeton University Press, Princeton, NJ (2008)

\bibitem{Alnes2015FEniCS}
Aln{\ae}s, M.S., Blechta, J., Hake, J., Johansson, A., Kehlet, B., Logg, A.,
  Richardson, C., Ring, J., Rognes, M.E., Wells, G.N.: The {FEniCS} project
  version 1.5.
\newblock Archive of Numerical Software \textbf{3}(100) (2015)

\bibitem{Blauth2021cashocs}
Blauth, S.: {cashocs: A Computational, Adjoint-Based Shape Optimization and
  Optimal Control Software}.
\newblock SoftwareX \textbf{13}, 100646 (2021)

\bibitem{Blauth2021Nonlinear}
Blauth, S.: {N}onlinear {C}onjugate {G}radient {M}ethods for {PDE}
  {C}onstrained {S}hape {O}ptimization {B}ased on {S}teklov-{P}oincaré-{T}ype
  {M}etrics.
\newblock SIAM J. Optim. \textbf{31}(2), 1658--1689 (2021)

\bibitem{Blauth2020Model}
Blauth, S., Leith\"{a}user, C., Pinnau, R.: Model hierarchy for the shape
  optimization of a microchannel cooling system.
\newblock ZAMM Z. Angew. Math. Mech. p. e202000166 (2020)

\bibitem{Blauth2020Shape}
Blauth, S., Leith\"{a}user, C., Pinnau, R.: Shape sensitivity analysis for a
  microchannel cooling system.
\newblock J. Math. Anal. Appl. \textbf{492}(2), 124476 (2020)

\bibitem{Dai1999nonlinear}
Dai, Y.H., Yuan, Y.: A nonlinear conjugate gradient method with a strong global
  convergence property.
\newblock SIAM J. Optim. \textbf{10}(1), 177--182 (1999)

\bibitem{Delfour2011Shapes}
Delfour, M.C., Zol\'{e}sio, J.P.: Shapes and geometries, \emph{Advances in
  Design and Control}, vol.~22, second edn.
\newblock Society for Industrial and Applied Mathematics (SIAM), Philadelphia,
  PA (2011)

\bibitem{Etling2020First}
Etling, T., Herzog, R., Loayza, E., Wachsmuth, G.: First and second order shape
  optimization based on restricted mesh deformations.
\newblock SIAM J. Sci. Comput. \textbf{42}(2), A1200--A1225 (2020)

\bibitem{Fletcher1964Function}
Fletcher, R., Reeves, C.M.: Function minimization by conjugate gradients.
\newblock Comput. J. \textbf{7}, 149--154 (1964)

\bibitem{Gangl2015Shape}
Gangl, P., Langer, U., Laurain, A., Meftahi, H., Sturm, K.: Shape optimization
  of an electric motor subject to nonlinear magnetostatics.
\newblock SIAM J. Sci. Comput. \textbf{37}(6), B1002--B1025 (2015)

\bibitem{Hager2005new}
Hager, W.W., Zhang, H.: A new conjugate gradient method with guaranteed descent
  and an efficient line search.
\newblock SIAM J. Optim. \textbf{16}(1), 170--192 (2005)

\bibitem{Hager2006survey}
Hager, W.W., Zhang, H.: A survey of nonlinear conjugate gradient methods.
\newblock Pac. J. Optim. \textbf{2}(1), 35--58 (2006)

\bibitem{Hestenes1952Methods}
Hestenes, M.R., Stiefel, E.: Methods of conjugate gradients for solving linear
  systems.
\newblock J. Research Nat. Bur. Standards \textbf{49}, 409--436 (1953) (1952)

\bibitem{Hinze2009Optimization}
Hinze, M., Pinnau, R., Ulbrich, M., Ulbrich, S.: Optimization with {PDE}
  constraints, \emph{Mathematical Modelling: Theory and Applications}, vol.~23.
\newblock Springer, New York (2009)

\bibitem{Hohmann2019Shape}
Hohmann, R., Leith\"{a}user, C.: Shape optimization of a polymer distributor
  using an {E}ulerian residence time model.
\newblock SIAM J. Sci. Comput. \textbf{41}(4), B625--B648 (2019)

\bibitem{Kelley1999Iterative}
Kelley, C.T.: Iterative methods for optimization, \emph{Frontiers in Applied
  Mathematics}, vol.~18.
\newblock Society for Industrial and Applied Mathematics (SIAM), Philadelphia,
  PA (1999)

\bibitem{Kriegl1997convenient}
Kriegl, A., Michor, P.W.: The convenient setting of global analysis,
  \emph{Mathematical Surveys and Monographs}, vol.~53.
\newblock American Mathematical Society, Providence, RI (1997)

\bibitem{Leithaeuser2021Energy}
Leith{\"a}user, C., Pinnau, R.: Energy-Efficient High Temperature Processes via
  Shape Optimization, pp. 123--143.
\newblock Springer International Publishing, Cham (2021)

\bibitem{Leithaeuser2018Designing}
Leith\"{a}user, C., Pinnau, R., Fe{\ss}ler, R.: Designing polymer spin packs by
  tailored shape optimization techniques.
\newblock Optim. Eng. \textbf{19}(3), 733--764 (2018)

\bibitem{Logg2012Automated}
Logg, A., Mardal, K.A., Wells, G.N., et~al.: Automated Solution of Differential
  Equations by the Finite Element Method.
\newblock Springer (2012)

\bibitem{Mueller2021novel}
M\"uller, P.M., K\"uhl, N., Siebenborn, M., Deckelnick, K., Hinze, M., Rung,
  P.: A novel $p$-harmonic descent approach applied to fluid dynamic shape
  optimization.
\newblock Structural and Multidisciplinary Optimization \textbf{64}, 3489--3503
  (2021)

\bibitem{Othmer2014Adjoint}
Othmer, C.: Adjoint methods for car aerodynamics.
\newblock J. Math. Ind. \textbf{4}, Art. 6, 23 (2014)

\bibitem{Paganini2021Fireshape}
Paganini, A., Wechsung, F.: Fireshape: a shape optimization toolbox for
  {F}iredrake.
\newblock Struct. Multidiscip. Optim. \textbf{63} (2021)

\bibitem{Polak1969Note}
Polak, E., Ribi\`ere, G.: Note sur la convergence de m\'{e}thodes de directions
  conjugu\'{e}es.
\newblock Rev. Fran\c{c}aise Informat. Recherche Op\'{e}rationnelle
  \textbf{3}(16), 35--43 (1969)

\bibitem{Ring2012Optimization}
Ring, W., Wirth, B.: Optimization methods on {R}iemannian manifolds and their
  application to shape space.
\newblock SIAM J. Optim. \textbf{22}(2), 596--627 (2012)

\bibitem{Schmidt2013Three}
Schmidt, S., Ilic, C., Schulz, V., Gauger, N.R.: Three-dimensional large-scale
  aerodynamic shape optimization based on shape calculus.
\newblock AIAA Journal \textbf{51}(11), 2615--2627 (2013)

\bibitem{Schulz2014Riemannian}
Schulz, V.H.: A {R}iemannian view on shape optimization.
\newblock Found. Comput. Math. \textbf{14}(3), 483--501 (2014)

\bibitem{Schulz2016Efficient}
Schulz, V.H., Siebenborn, M., Welker, K.: Efficient {PDE} constrained shape
  optimization based on {S}teklov-{P}oincar\'{e}-type metrics.
\newblock SIAM J. Optim. \textbf{26}(4), 2800--2819 (2016)

\bibitem{Sturm2015Shape}
Sturm, K.: Shape differentiability under non-linear {PDE} constraints.
\newblock In: New trends in shape optimization, \emph{Internat. Ser. Numer.
  Math.}, vol. 166, pp. 271--300. Birkh\"{a}user/Springer, Cham (2015)

\end{thebibliography}

\end{document}